\documentclass{amsart}
\usepackage{latexsym,amsxtra,amscd,ifthen}
\usepackage{amsfonts}
\usepackage{color,graphicx}
\usepackage{verbatim}
\usepackage{amsmath}
\usepackage{amsthm}
\usepackage{amssymb}
\usepackage{url}

\usepackage{tikz}
\usetikzlibrary{arrows,shapes,chains}

\theoremstyle{plain}

\newtheorem{theorem}{Theorem}
\newtheorem{lemma}[theorem]{Lemma}
\newtheorem{proposition}[theorem]{Proposition}
\newtheorem{corollary}[theorem]{Corollary}

\numberwithin{theorem}{section}
\numberwithin{equation}{theorem}

\theoremstyle{definition}
\newtheorem{definition}[theorem]{Definition}
\newtheorem{example}[theorem]{Example}
\newtheorem{remark}[theorem]{Remark}
\newtheorem{question}[theorem]{Question}
\newtheorem*{question*}{Question}

\newcommand{\iso}{\xrightarrow{\,\smash{\raisebox{-0.65ex}{\ensuremath{\scriptstyle\sim}}}\,}}

\DeclareMathOperator{\gr}{gr}

\DeclareMathOperator{\GKdim}{GKdim}

\DeclareMathOperator{\Kdim}{Kdim}

\DeclareMathOperator{\Htr}{Htr}
\DeclareMathOperator{\injdim}{injdim}

\begin{document}

\title{Cancellation of Morita and Skew Types}

\author{Xin Tang, James J. Zhang and Xiangui Zhao}

\address{Tang: Department of Mathematics \& Computer Science, 
Fayetteville State University, Fayetteville, NC 28301,
USA}

\email{xtang@uncfsu.edu}

\address{Zhang: Department of Mathematics, Box 354350,
University of Washington, Seattle, Washington 98195, USA}

\email{zhang@math.washington.edu}

\address{Zhao: Department of Mathematics,
Huizhou University,
Huizhou 516007 Guandong, China}

\email{zhaoxg@hzu.edu.cn}

\begin{abstract}
We study both Morita cancellative and skew cancellative properties 
of noncommutative algebras as initiated recently in several papers 
and explore that which classes of noncommutative algebras are 
Morita cancellative (respectively, skew cancellative). Several new 
results concerning these two types of cancellations, as well as the 
classical cancellation, are proved.
\end{abstract}





\maketitle


\section*{Introduction}
\label{xxsec0}

Let $\Bbbk$ denote a base field. A $\Bbbk$-algebra $A$ is said to be 
{\it cancellative} in the category of $\Bbbk$-algebras if any 
$\Bbbk$-algebra isomorphism $\phi\colon A[x]\cong B[x]$ (where $B$ is 
another $\Bbbk$-algebra) implies that $A$ is isomorphic to $B$ as a 
$\Bbbk$-algebra. Geometrically, a $\Bbbk$-variety $V$ is called 
{\it cancellative} if any isomorphism $V\times {\mathbb A}^1
\cong W\times {\mathbb A}^1$ for another $\Bbbk$-variety $W$ implies
that $V\cong W$. Cancellative properties have been extensively 
investigated for commutative domains, especially for the commutative 
polynomial rings, in the literature \cite{AEH, Cr, CM, Fu, Ma1, MS, Ru}. 
Note that not every commutative domain is cancellative \cite{Da, Fi, Ho}.
In the commutative case, we sometimes call this kind of question 
``Zariski cancellation problem", as the cancellation problem of fields was 
first raised by Zariski in 1949 \cite{Se}. In the noncommutative case, 
the study of cancellative properties dates back to the early 1970s 
\cite{AEH, As, BR, CE, EH, EK}. Despite great success achieved in the 
work of Gupta \cite{Gu1, Gu2}, the Zariski cancellation problem still 
remains open for the commutative polynomial ring 
$\Bbbk[t_{1}, \cdots, t_{n}]$ with $n\geq 3$ in the characteristic 
zero case, see \cite{Kr, Gu3} for a history of this open problem. 

Recently, the study of cancellation problem has been revitalized for 
noncommutative algebras thanks to \cite{BZ1}, which mainly employs 
the famous Makar-Limanov invariants \cite{Ma1, Ma2} and the 
noncommutative discriminants as investigated in \cite{CPWZ1, CPWZ2}. 
It is usually very difficult to describe the discriminant for a given algebra; 
fortunately, many useful results on discriminants have been further 
established in \cite{BY, CYZ1, CYZ2, GKM, GWY, NTY, WZ}. Ever since 
\cite{BZ1}, there has been much progress made in the study of 
cancellation problem for noncommutative algebras 
\cite{BZ2, CYZ1, Ga, LR, LY, LeWZ, LuWZ, LMZ, Ta1, Ta2, TRZ}.
In particular, the cancellative property was established in 
\cite{LeWZ} for many classes of algebras which are not necessarily 
domains; and the Morita cancellation and derived cancellation were 
introduced and studied for algebras in \cite{LuWZ}. The cancellative 
properties for Poisson algebras were most recently examined in \cite{GaW}. 

The problem of skew cancellations was early considered in \cite{AKP} and 
revisited in recent papers \cite{Be, BHHV}. One of the main motivations of 
this paper is to introduce the multi-variable version of the skew 
cancellative property.
In particular, we study the following two closely related topics:
\begin{enumerate}
\item[(1)]
Strong Morita cancellation as initiated in \cite{LuWZ}.
\item[(2)]
Multi-variable version of the skew cancellation as initiated in 
\cite{AKP, Be, BHHV}.
\end{enumerate}
Although we are mainly interested in the Morita cancellation as introduced in 
\cite{LuWZ}, we also make some comments on the derived cancellation
in Section \ref{xxsec5}. Our ideas and methods are inspired by the ones 
in \cite{AEH, BR, CE, BZ1, CYZ1, CYZ2, LeWZ, LuWZ}. 

Before we state our results, we need to recall a list of basic
definitions about the Morita cancellation from \cite{BZ1, LeWZ, LuWZ}.
Later we will recall another list of definitions concerning the skew 
cancellation. We denote by $A[t_{1}, \cdots, t_{n}]$ the polynomial 
extension of an algebra $A$ with commuting multi-variables 
$t_1,\cdots,t_n$ and by $M(A)$ the category of all 
right $A$-modules. All algebraic objects are defined over the 
base field $\Bbbk$. 

\begin{definition}
\label{xxdef0.1}
Let $A$ be an algebra.
\begin{enumerate}
\item[(1)]
We say $A$ is {\it strongly cancellative} if 
any $\Bbbk$-algebra isomorphism 
$$A[s_{1}, \cdots, s_{n}] \cong B[t_{1}, \cdots, t_{n}],$$ 
for any $n\geq 1$ and any algebra $B$, implies 
that $A$ is isomorphic to $B$ as a $\Bbbk$-algebra.
\item[(2)]
We say $A$ is {\it universally cancellative} 
if, for every finitely generated commutative domain $R$ 
with an ideal $I\subset R$ such that $\Bbbk\longrightarrow R
\longrightarrow R/I$ is an isomorphism and every algebra 
$B$, any $\Bbbk$-algebra isomorphism 
$$A\otimes_{\Bbbk} R\cong B\otimes_{\Bbbk} R$$ 
implies that $A\cong B$ as $\Bbbk$-algebras.
\end{enumerate}
\end{definition}

\begin{definition}
\label{xxdef0.2}
Let $A$ be an algebra.
\begin{enumerate}
\item[(1)]
We say $A$ is {\it Morita cancellative} if any 
equivalence of abelian categories 
$$M(A[s])\cong M(B[t]),$$ 
for any algebra $B$, implies an 
equivalence of abelian categories
$$M(A)\cong M(B).$$ 
\item[(2)]
We say $A$ is {\it strongly Morita cancellative} if any 
equivalence of abelian categories 
$$M(A[s_{1}, \cdots, s_{n}])\cong M(B[t_{1}, \cdots, t_{n}]),$$ 
for any $n\geq 1$ and any algebra $B$, implies an 
equivalence of abelian categories
$$M(A)\cong M(B).$$ 
\end{enumerate}
\end{definition}

The above (strong) Morita cancellation of noncommutative algebras is a natural 
generalization of the classical cancellation in the category of commutative algebras. 
The following universal version of the Morita cancellation is similar to those in 
Definition \ref{xxdef0.1}. 

\begin{definition}
\label{xxdef0.3}
Let $A$ be an algebra.
We say $A$ is {\it universally Morita 
cancellative} if, for every finitely generated commutative 
domain $R$ with an ideal $I\subset R$ such that 
$\Bbbk\longrightarrow R\longrightarrow R/I$ is an isomorphism 
and every algebra $B$, any equivalence of abelian 
categories 
$$M(A\otimes_{\Bbbk} R)\cong M(B\otimes_{\Bbbk} R)$$ 
implies an equivalence of abelian categories
$$M(A)\cong M(B).$$ 
\end{definition}

We refer the reader to \cite{BZ1, LeWZ, LuWZ} and Section 1 for other basic 
definitions. Now we can state our results about the Morita cancellation. 
The first one is a Morita version of \cite[Proposition 1.3]{BZ1}. 

\begin{theorem}
\label{xxthm0.4}
Let $A$ be an algebra with center being the base field 
$\Bbbk$. Then $A$ is universally Morita cancellative.
\end{theorem}

The next result can be viewed as both a Morita version and a strengthened 
version of a partial combination of \cite[Theorem 4.1]{LeWZ} with \cite[Theorem 4.2]{LeWZ}. 
We denote the nilradical of an algebra $A$ by $N(A)$. The definition of the strongly retractable 
property is given in Definition \ref{xxdef1.1} (see also \cite[Definition 2.1]{LeWZ}).

\begin{theorem}
\label{xxthm0.5}
Let $A$ be an algebra with center $Z$ such that either $Z$ or $Z/N(Z)$ is strongly 
retractable {\rm{(}}respectively, strongly detectable{\rm{)}}, then $A$ is strongly 
cancellative and strongly Morita cancellative.
\end{theorem}

This theorem has several consequences. For example, by 
using Theorem \ref{xxthm0.5} (and combining with Lemma 
\ref{xxlem1.2}(2)), the hypotheses of being ``strongly 
Hopfian'' in \cite[Theorem 0.3, Lemma 3.6, Theorem 4.2(2), 
Corollary 4.3, Corollary 7.3]{LuWZ} and \cite[Theorem 0.2, 
Theorem 4.2]{LeWZ} are superfluous. Next we give an explicit 
application. Recall that a commutative algebra is called
{\it von Neumann regular} if it is reduced and has Krull dimension zero.

\begin{corollary}
\label{xxcor0.6}
Let $A$ be an algebra with center $Z$. 
\begin{enumerate}
\item[(1)]
If $Z/N(Z)$ is generated by a set of units of $Z/N(Z)$, 
then $Z$ and $A$ are strongly cancellative and strongly Morita 
cancellative.
\item[(2)]
If $Z/N(Z)$ is a von Neumann regular algebra, then $Z$ and $A$ are strongly 
cancellative and strongly Morita cancellative.
\item[(3)]
If $Z$ is a finite direct sum of local algebras, then $Z$ and $A$ 
are strongly cancellative and strongly Morita 
cancellative.
\end{enumerate}
\end{corollary}

Note that Corollary \ref{xxcor0.6}(2) answers 
\cite[Question 0.1]{LeWZ} positively. All statements concerning
the Morita cancellation in Corollary \ref{xxcor0.6} are new. The 
above corollary also has many applications in practice. 

The second part of the paper deals with the skew cancellation which 
is another natural generalization of the classical cancellation. Here we 
replace the polynomial extensions by the Ore extensions. Let $A$ be 
an algebra. Let $\sigma$ be an algebra automorphism of $A$ and $\delta$ 
be a $\sigma$-derivation of $A$. Then one can form the Ore extension, 
denoted by $A[t;\sigma,\delta]$, which shares many nice properties with 
the polynomial extension $A[t]$. The reader is referred to \cite[Chapter 1]{MR} 
for more details. We say $\sigma$ is {\it locally algebraic} if every finite
dimensional subspace of $A$ is contained in a $\sigma$-stable 
finite dimensional subspace of $A$. It is obvious that the identity map 
is locally algebraic. An iterated Ore extension of $A$ is of the form 
$$A[t_1;\sigma_1,\delta_1][t_2;\sigma_2, \delta]\cdots 
[t_n;\sigma_n,\delta_n]$$
where $\sigma_{i}$ is an algebra automorphism of 
$A_{i-1}:=A[t_1;\sigma_1,\delta_1]\cdots [t_{i-1};\sigma_{i-1},\delta_{i-1}]$
and $\delta_i$ is a $\sigma_i$-derivation of $A_{i-1}$.

\begin{definition}
\label{xxdef0.7}
Let $A$ be an algebra.
\begin{enumerate}
\item[(1)]
We say $A$ is {\it skew cancellative} if any 
isomorphism of algebras
$$A[t;\sigma,\delta]\cong A'[t';\sigma',\delta']$$ 
for another algebra $A'$, implies an isomorphism of algebras
$$A\cong A'.$$ 
\item[(2)]
We say $A$ is {\it strongly skew cancellative} if any 
isomorphism of algebras
$$A[t_1;\sigma_1,\delta_1]\cdots [t_n;\sigma_n,\delta_n]
\cong A'[t'_1;\sigma'_1,\delta'_1]\cdots [t'_n;\sigma'_n,\delta'_n]$$ 
for any $n\geq 1$ and any algebra $A'$, implies an isomorphism of algebras
$$A\cong A'.$$ 
\end{enumerate}
\end{definition}

Occasionally, we will restrict our attention to special types of Ore extensions and/or 
special classes of base algebras. For example, we make the following definition.

\begin{definition}
\label{xxdef0.8}
Let $A$ be an algebra.
\begin{enumerate}
\item[(1)]
We say $A$ is {\it $\sigma$-cancellative} if in Definition 
\ref{xxdef0.7}(1), only Ore extensions with 
$\delta=0$ and $\delta'=0$ are considered.
We say $A$ is {\it strongly $\sigma$-cancellative} if in Definition 
\ref{xxdef0.7}(2), only Ore extensions with 
$\delta_i=0$ and $\delta'_i=0$, for all $i$, are considered.
\item[(2)]
We say $A$ is {\it $\delta$-cancellative} if in Definition 
\ref{xxdef0.7}(1), only Ore extensions with 
$\sigma=Id_{A}$ and $\sigma'=Id_{A'}$ are considered.
We say $A$ is {\it strongly $\delta$-cancellative} if in Definition 
\ref{xxdef0.7}(2), only Ore extensions with 
$\sigma_i=Id$ and $\sigma'_i=Id$, for all $i$, are considered.
\item[(3)]
We say $A$ is {\it $\sigma$-algebraically cancellative} if in Definition 
\ref{xxdef0.7}(1), only Ore extensions with locally algebraic 
$\sigma$ and $\sigma'$ are considered.
We say $A$ is {\it strongly $\sigma$-algebraically cancellative} if in Definition 
\ref{xxdef0.7}(2), only Ore extensions with locally algebraic 
$\sigma_i$ and $\sigma'_i$ are considered.
\end{enumerate}
\end{definition}

A classical cancellation problem is equivalent to a skew cancellation 
problem with $(\sigma, \delta)=(Id,0)$. Therefore, the skew (or 
$\sigma$-, or $\delta$-)cancellation is a natural extension and a 
strictly stronger version of the classical cancellation. It follows 
from the definition that the $\sigma$-algebraically cancellative
property is stronger than the $\delta$-cancellative property. See
Figure 1 after Example \ref{xxex5.5}. The $\delta$-cancellation 
was first considered in \cite{AKP}, and then in \cite{Be}. In 
\cite[Theorem 1.2]{BHHV}, a very nice result concerning both 
$\sigma$- and $\delta$-cancellations was proved, however, the skew 
cancellative property remains open. As remarked in \cite{BHHV}, 
``{\it would be interesting to give a `unification' of 
the two results occurring in \cite[Theorem 1.2]{BHHV} and 
prove that skew cancellation holds for general skew polynomial
extensions, although this appears to be considerably more 
subtle than the cases we consider.}'' One of our main goals in the 
second half of the paper is to introduce a unified approach to the 
skew cancellation problem (including both $\sigma$- and $\delta$-
cancellation). 

To state our main results, we need to recall the definition of 
a divisor subalgebra as introduced in \cite{CYZ1}. Let $A$ be a 
domain. Let $F$ be a subset of $A$. Let $Sw(F)$ 
denote the set of $g\in A$ such that $f=agb$ for some $a, b\in A$ and 
$0\neq f\in F$. That is, $Sw(F)$ consists of all the subwords of the 
elements in $F$. We set $D_{0}(F)=F$ and inductively define $D_{n}(F)$ for $n\geq 1$ as the 
$\Bbbk$-subalgebra of $A$ generated by $Sw(D_{n-1}(F))$. The 
subalgebra ${\mathbb D}(F)=\bigcup_{n\geq 0} D_{n}(F)$ is called the {\it divisor 
subalgebra} of $A$ generated by $F$. If $F$ is the singleton $\{f\}$, 
we simply write ${\mathbb D}(\{f\})$ as ${\mathbb D}(f)$. See Section 5 
for more details.

\begin{theorem}
\label{xxthm0.9} 
Let $A$ be an affine domain of finite GKdimension. Suppose that
${\mathbb D}(1)=A$. Then $A$ is strongly $\sigma$-algebraically 
cancellative. As a consequence, it is strongly $\delta$-cancellative.
\end{theorem}

To prove several classes of algebras are skew cancellative, we need
to use a structure result of division algebras. Recall from \cite{Sc} that a 
simple artinian ring $S$ is {\it stratiform} over $\Bbbk$ if there is a chain
of simple artinian rings
$$S = S_n \supseteq S_{n-1} \supseteq \cdots \supseteq S_1 \supseteq S_0 =\Bbbk$$
where, for every $i$, either
\begin{enumerate}
\item[(i)] 
$S_{i+1}$ is finite over $S_{i}$ on both sides; or
\item[(ii)] 
$S_{i+1}$ is equal to the quotient ring of the 
Ore extension $S_i[t_i; \sigma_i, \delta_i]$ for 
an automorphism $\sigma_i$ of $S_i$ and $\sigma_i$-derivation 
$\delta_i$ of $S_i$.
\end{enumerate}
Such a chain of simple artinian rings is called a stratification of 
$S$. The {\it stratiform length} of $S$ is the number of steps in 
the chain that are of type (ii). An important fact established in 
\cite{Sc} is that the stratiform length is an invariant of $S$. 
A Goldie prime ring $A$ is called {\it stratiform} if the quotient 
division ring of $A$, denoted by $Q(A)$, is stratiform.

\begin{theorem}
\label{xxthm0.10} 
Let $A$ be a noetherian domain that is stratiform. Suppose 
that ${\mathbb D}(1)=A$. Then $A$ is strongly skew cancellative
in the category of noetherian stratiform domains.
\end{theorem}

The following algebras are stratiform with ${\mathbb D}(1)=A$.
As a result, they are skew cancellative.
\begin{enumerate}
\item[(a)]
Quantum torus or quantum Laurent polynomial algebras given 
in Example \ref{xxex4.3}(5),
\item[(b)]
Localized quantum Weyl algebras $B^q_1(\Bbbk)$ in 
Example \ref{xxex4.3}(2). 
\item[(c)]
Affine commutative domain $A$ of GKdimension one satisfying
$A^{\times}\supsetneq \Bbbk^{\times}$ [Lemma \ref{xxlem4.4}(9)].
\item[(d)]
Any noetherian domain that can be written as a finite tensor product 
(resp. some version of a twisted tensor product) of the algebras in parts (a,b,c).
\end{enumerate}

We further prove a few results concerning the strongly cancellation, 
the strongly Morita cancellation, and the skew cancellation of noncommutative 
algebras, see Theorems \ref{xxthm4.6} and \ref{xxthm5.4}, and Proposition \ref{xxpro5.8}.

The paper is organized as follows. Section 1 reviews some basic materials. 
Section 2 concerns the Morita cancellative property where Theorem 
\ref{xxthm0.5} and Corollary \ref{xxcor0.6} are proven. In Section 3
we recall some basic properties about the Gelfand-Kirillov dimension 
and homological transcendence degree of noncommutative algebras. 
Then we review the definition of a divisor subalgebra and study the
skew cancellative property in Section 4. We also prove our 
main results, namely, Theorems \ref{xxthm0.9} and
\ref{xxthm0.10} in Section 4. The final section contains some 
comments, examples, remarks and questions.

\section{Preliminaries}
Throughout $\Bbbk$ denotes a base field. All algebras are $\Bbbk$-algebras 
and all algebra homomorphisms are $\Bbbk$-linear algebra homomorphisms. As needed, we will 
continue to use the notation and convention introduced in \cite{BZ1, LeWZ, LuWZ}.

We only recall a small selected set of definitions. 

\begin{definition}
\cite[Definition 2.1]{LeWZ}
\label{xxdef1.1} 
Let $A$ be an algebra.
\begin{enumerate}
\item[(1)]
We say $A$ is {\it retractable} if, for any algebra $B$, any algebra 
isomorphism 
$$\phi : A[s] \cong B[t]$$
implies that $\phi(A) = B$.
\item[(2)]
\cite[p.311]{AEH}
We say $A$ is {\it strongly retractable} if, for any algebra $B$ and 
integer $n \geq 1$, any
algebra isomorphism 
$$\phi : A[s_1,\cdots, s_n] \cong B[t_1,\cdots, t_n]$$ 
implies that $\phi(A) = B$.
\end{enumerate}
\end{definition}

The following lemma of Brewer-Rutter \cite{BR} is useful. 

\begin{lemma} \cite[Lemma 1]{BR}
\label{xxlem1.2}
Let $A$ be an algebra with center $Z$. 
\begin{enumerate}
\item[(1)] 
If $f_{1}, \cdots, f_{n}$ are $Z$-generators of the polynomial ring 
$Z[Y_{1}, \cdots, Y_{n}]$, then the $A$-endomorphism $\tau$ of 
$A[Y_{1}, \cdots, Y_{n}]$ defined by $\tau(Y_{i})=f_{i}$ for each 
$1\leq i\leq n$ is an isomorphism.
\item[(2)]
As a special case, if $A$ is commutative, and if 
$f_{1}, \cdots, f_{n}$ are $A$-generators of the polynomial ring 
$A[Y_{1}, \cdots, Y_{n}]$, then $A\{f_1,\cdots,f_n\}=
A[f_{1}, \cdots, f_{n}]$.
\end{enumerate}
\end{lemma}

For any algebra $A$, let $Z(A)$ or simply $Z$ denote the center of $A$ 
and let $N(A)$ denote the nilpotent radical of $A$. Suppose two algebras
$R$ and $S$ are Morita equivalent. Let 
\begin{equation}
\label{E1.2.1}\tag{E1.2.1}
\omega: Z(R)\to Z(S)
\end{equation}
be the isomorphism of the centers given in \cite[Lemma 1.2(3)]{LuWZ}. 
Note that we can use all facts listed in \cite[Lemma 1.2(3)]{LuWZ}.

In the following two definitions, we have the following abbreviations.
$${\text{$S=$strongly, $M=$Morita, and $R=$reduced.
}}$$

\begin{definition}
\label{xxdef1.3}
Let $A$ be an algebra.
\begin{enumerate}
\item[(1)]
We say $A$ is {\it Morita $Z$-detectable} or simply 
{\it MZ-detectable} if, 
for any algebra $B$ and any equivalence of abelian categories 
$$\mathcal{E}\colon M(A[s])\longrightarrow M(B[t]),$$ 
with the induced isomorphism, see \eqref{E1.2.1},
$$\omega \colon Z(A[s])(=Z(A)[s])\longrightarrow Z(B[t])(=Z(B)[t]),$$ 
implies that 
$$Z(B)[t]=Z(B)\{\omega(s)\}.$$ 
By Lemma \ref{xxlem1.2}, we actually have that 
$$Z(B)[t]=Z(B)[\omega(s)].$$ 
\item[(2)]
We say $A$ is {\it strongly Morita $Z$-detectable} or simply 
{\it SMZ-detectable} if, for each 
$n\geq 1$ and any algebra $B$, any equivalence of abelian categories 
$$\mathcal{E}\colon M(A[s_{1}, \cdots, s_{n}])\longrightarrow 
M(B[t_{1}, \cdots, t_{n}])$$ 
implies that, with $\omega$ given in \eqref{E1.2.1} for algebras
$R=A[s_{1}, \cdots, s_{n}]$ and $S=B[t_{1}, \cdots, t_{n}]$,
$$Z(B)[t_{1}, \cdots, t_{n}]
=Z(B)\{\omega(s_{1}), \cdots, \omega(s_{n})\}.$$
Once again, by Lemma \ref{xxlem1.2}, we actually have that
$$Z(B)[t_{1}, \cdots, t_{n}]
=Z(B)[\omega(s_{1}), \cdots, \omega(s_{n})].$$
\end{enumerate} 
\end{definition}

In the next definition, $\omega$ is given as in \eqref{E1.2.1}
for appropriate $R$ and $S$ and $\overline{\omega}$ is 
an induced isomorphism in appropriate setting. 

\begin{definition}
\label{xxdef1.4}
Let $A$ be an algebra.
\begin{enumerate}
\item[(1)]
We say $A$ is {\it reduced Morita $Z$-detectable} or {\it RMZ-detectable} 
if, for any algebra $B$ and any equivalence of abelian categories 
$$E\colon M(A[s])\longrightarrow M(B[t]),$$ 
with the induced isomorphism (modulo prime radicals): 
$$
\overline{\omega}\colon Z(A)/N(Z(A))[\overline{s}]
\longrightarrow Z(B)/N(Z(B))[\overline{t}],
$$
implies that 
$$Z(B)/N(Z(B))[\overline{t}]=Z(B)/N(Z(B))
[\overline{\omega}(\overline{s})].$$ 
\item[(2)]
We say $A$ is {\it strongly reduced Morita $Z$-detectable} or simply
{\it SRMZ-detectable} if, for each $n\geq 1$ and any algebra $B$, any 
equivalence of abelian categories 
$$\mathcal{E}\colon M(A[s_{1}, \cdots, s_{n}])
\longrightarrow M(B[t_{1}, \cdots, t_{n}])$$ 
implies that 
$$
Z(B)/N(Z(B))[\overline{t}_{1}, \cdots, \overline{t}_{n}]=
Z(B)/N(Z(B))[\overline{\omega}(\overline{s}_{1}), \cdots, 
\overline{\omega}(\overline{s}_{n})].
$$
\end{enumerate}
\end{definition}

Several retractabilities are defined in \cite{LeWZ, LuWZ}. It has been observed in 
\cite{BR, LeWZ, LuWZ} that the cancellative property of an algebra $A$ is controlled 
by its center $Z(A)$ to a large degree. In the rest of this section, we establish or recall 
some basic facts. In Section $2$, we will show that there is a Morita analogue of 
\cite[Theorem 1]{BR} and \cite[Theorem 4.2]{LeWZ} can be strengthened. 

The following result is essentially verified in the proof of \cite[Theorem 1]{BR}, 
see \cite[pp. 485--486]{BR}, and in the proof of 
\cite[Statement $\#4$, pp. 334--335]{EK}. For reader's 
convenience, we recall it as a lemma and reproduce its proof as follows.

\begin{lemma} \cite[Theorem 1]{BR}
\label{xxlem1.5}
Suppose that $A$ and $B$ are commutative algebras. Let 
$$\sigma \colon A[s_{1}, \cdots, s_{n}]\longrightarrow 
B[t_{1}, \cdots, t_{n}]$$ be an isomorphism of algebras such that 
the induced isomorphism modulo prime radicals, denoted by
\[
\overline{\sigma} 
\colon A/N(A)[\overline{s}_{1}, \cdots, \overline{s}_{n}]
\longrightarrow B/N(B)[\overline{t}_{1}, \cdots, \overline{t}_{n}]
\]
has the property that 
$$B/N(B)[\overline{t}_{1}, \cdots, \overline{t}_{n}]=
B/N(B)\{\overline{f}_{1}, \cdots, \overline{f}_{n}\}$$ 
where $f_{i}=\sigma(s_{i})$ for $i=1, \cdots, n$. Then 
$$B[t_{1}, \cdots, t_{n}]=B\{f_{1}, \cdots, f_{n}\}
=B[f_{1}, \cdots, f_{n}]$$
where $f_{1}, \cdots, f_{n}$ are considered as 
commutative indeterminates over $B$.
\end{lemma}

In essence, Lemma \ref{xxlem1.5} implies that a certain 
detectability lifts from $A/N(A)$ to $A$. 

\begin{proof}[Proof of Lemma \ref{xxlem1.5}] 
Since $B/N(B))[\overline{t}_{1}, \cdots, \overline{t}_{n}]=
B/N(B)\{\overline{f}_{1}, \cdots, \overline{f}_{n}\}$, there 
are polynomials in $B$, say $g_1, \cdots, g_n$, such that 
\[
\overline{t}_{i}=
\overline{g}_{i}(\overline{f}_{1}, \cdots, \overline{f}_{n}).
\]
As a result, for $i=1, \cdots, n$, we have the following
\[
t_{i}=g_{i}(f_{1}, \cdots, f_{n})+h_{i}(t_{1}, \cdots, t_{n})
\]
where $h_{i}\in N(B[t_{1},\cdots, t_{n}])=
N(B)[t_{1}, \cdots, t_{n}]$. Denote by $N_{0}$ the ideal of $B$ 
generated by the coefficients of $h_{1}, \cdots, h_{n}$. Then, 
by induction, we have that 
\[
B[t_{1}, \cdots, t_{n}]=B\{f_{1}, \cdots, f_{n}\}
+N_{0}^{m}B[t_{1}, \cdots, t_{n}]
\]
for each $m\geq 1$. Since $N_{0}$ is a finitely generated 
ideal of $B$ and $N_{0}$ is contained in $N(B)$, the 
prime radical of $B$, we have that $N_{0}$ is nilpotent. 
As a result, we have that 
\[
B[t_{1}, \cdots, t_{n}]=B\{f_{1}, \cdots, f_{n}\}.
\]
Using Lemma \ref{xxlem1.2}(2), we conclude that 
$$B[t_{1}, \cdots, t_{n}]=B[f_{1}, \cdots, f_{n}]$$ 
where the elements $f_{1}, \cdots, f_{n}$ are regarded as 
commutative indeterminates over $B$.
\end{proof}

We now state a couple of easy facts about detectability.

\begin{lemma}
\label{xxlem1.6}
Let $Z$ be the center of an algebra $A$.
\begin{enumerate}
\item[(1)]
If $Z$ is strongly retractable, then $A$ is strongly Morita 
$Z$-retractable, and consequently, SMZ-detectable. 
\item[(2)]
Suppose that $Z/N(Z)$ is strongly retractable. Then $A$ is strongly 
reduced Morita $Z$-retractable. As a consequence, $A$ is 
SMZ-detectable. 
\end{enumerate}
\end{lemma}

\begin{proof}
(1) The first assertion follows from \cite[Definition 2.6(4)]{LeWZ}.
For the second assertion, see the proof of \cite[Lemma 3.4]{LeWZ}.

(2) The first statement is part (1). By part (1), 
$A$ is strongly reduced Morita detectable. By Lemma \ref{xxlem1.5},
$A$ is strongly Morita $Z$-detectable, or SMZ-detectable.
\end{proof}

\begin{lemma}
\label{xxlem1.7}
Let $Z$ be the center of an algebra $A$.
\begin{enumerate}
\item[(1)]
Suppose that $A$ is either strongly Morita $Z$-retractable or 
strongly reduced Morita $Z$-retractable. Then $A$ is SMZ-detectable. 
\item[(2)] \cite[Theorem 1.2]{As}
If $A$ is reduced, then $A$ is SMZ-detectable if 
and only if $A$ is strongly Morita $Z$-retractable. 
\end{enumerate}
\end{lemma}

\begin{proof} (1) It follows from Lemma \ref{xxlem1.6}.

(2) If $A$ is strongly Morita $Z$-retractable, by the proof of 
\cite[Lemma 3.4]{LeWZ}, $A$ is SMZ-detectable. The converse statement
follows from the proof of \cite[Theorem 1.2]{As} which 
we repeat next.

Suppose $B$ is another algebra such that 
\[
\mathcal{E}\colon M(A[s_{1}, \cdots, s_{n}])
\longrightarrow M(B[t_{1}, \cdots, t_{n}])
\]
is an equivalence of abelian categories. Let 
$$\omega \colon Z(A)[s_{1}, \cdots, s_{n}]
\longrightarrow Z(B)[t_{1}, \cdots, t_{n}]$$ 
be the corresponding induced isomorphism given 
in \eqref{E1.2.1}. Denote by $f_{i}$ the element 
$\omega(s_{i})\in Z(B)[t_{1}, \cdots, t_{n}]$ for 
$i=1, \cdots, n$. Since $A$ 
is SMZ-detectable, by definition, 
$$Z(B)[t_{1}, \cdots, t_{n}]
=Z(B)[f_{1}, \cdots, f_{n}].$$ 
As a consequence, we have that 
$$\omega(Z(A))[f_{1}, \cdots, f_{n}]=
\omega(Z(A)[s_1,\cdots,s_n])
=Z(B)[t_{1}, \cdots, t_{n}]=
Z(B)[f_{1}, \cdots, f_{n}].$$ 
Now we need to show that 
$\omega(Z(A))=Z(B)$. To simplify the notation, we will 
denote $\omega(Z(A))$ by $R$ and $Z(B)$ by $S$ respectively, 
and $f_{i}$ by $X_{i}$ instead. Set 
$S_{k}=S[X_{1}, \cdots, X_{k-1}, X_{k+1}, \cdots, X_{n}]$ 
for $k=1, \cdots, n$. Then $S[X_{1}, \cdots, X_{n}]=S_{k}[X_{k}]$ 
is a polynomial algebra in a single indeterminate $X_{k}$ over $S_{k}$. 
Note that any element $\alpha$ of $R=\omega(Z(A))$ can be written 
in the following form: 
\[
\alpha=\beta_{0}+\beta_{1}X_{k}+\cdots + \beta_{m} X_{k}^{m}
\]
where $\beta_{i} \in S_{k}$. Suppose that 
$f(X_{k})=\gamma_{0}+\gamma_{1}X_{k}+\cdots \gamma_{l} X_{k}^{l}$ 
is a polynomial in $S_{k}[X_{k}]$ such that 
$S_{k}[X_{k}]=S_{k}[f(X_{k})]$. Then it is true that 
$\gamma_{1}$ is a unit and $\gamma_{2}, \cdots, \gamma_{l}$ 
are nilpotent elements of $S_{k}$. Since $A$ is reduced, 
its center $Z(A)$ is reduced. Then $R=\omega(Z(A))$ is reduced as well. 
As a result, $S=Z(B)$ is reduced. Thus, $S_{k}$ is reduced too. 
We have that $\gamma_{2}=\cdots=\gamma_{l}=0$. Note that 
$$\begin{aligned}
R[X_{1}, \cdots, X_{n}]&=
R[X_{1}, \cdots, X_{k-1}, X_{k}+\alpha, X_{k+1}, \cdots, X_{n}]\\
&=R[X_{1}, \cdots, X_{k-1}, X_{k}+\alpha^{2}, 
X_{k+1}, \cdots, X_{n}].
\end{aligned}
$$ 
As a result, we have that 
$$S_{k}[X_{k}]=S_{k}[X_{k}+\alpha]=S_{k}[X_{k}+\alpha^{2}]$$
which implies that $\beta_{1}, \cdots, \beta_{m}$ are nilpotent 
elements of $S_{k}$ and thus equal to zero. So we have that 
$\alpha\in S_{k}$ for $k=1, \cdots, n$. Note that $\bigcap_{k=1}^{n} S_{k}=S$. 
So we have proved that $\alpha \in S$ as desired. Note that 
$R\subseteq S$ and $R[X_{1}, \cdots, X_{n}]=S[X_{1}, \cdots, X_{n}]$ 
can imply that $R=S$ by \cite[Lemma 2]{BR}. 
\end{proof}

\begin{lemma}
\label{xxlem1.8}
An algebra $A$ is SMZ-detectable if and only if it is SRMZ-detectable.
\end{lemma}

\begin{proof}
Suppose that the algebra $A$ is SMZ-detectable and let 
\[
\mathcal{E}\colon M(A[s_{1}, \cdots, s_{n}])
\longrightarrow M(B[t_{1}, \cdots, t_{n}])
\]
be an equivalence of abelian categories where $B$ is another 
algebra. Note that the equivalence $\mathcal{E}$ induces an 
algebra isomorphism 
$$\omega \colon Z(A)[s_{1}, \cdots, s_{n}]\longrightarrow 
Z(B)[t_{1}, \cdots, t_{n}],$$ 
as in \eqref{E1.2.1}. Since $A$ is SMZ-detectable, 
we have that 
$$Z(B)[t_{1}, \cdots, t_{n}]=Z(B)[f_{1}, \cdots, f_{n}]$$ 
where $f_{i}=\omega(s_{i})$ for $i=1, \cdots, n$. Modulo both 
sides by the nil-radical, we obtain that
$$Z(B)/N(Z(B))[\overline{t}_{1}, \cdots, \overline{t}_{n}]
=Z(B)/N(Z(B))[\overline{f}_{1}, \cdots, \overline{f}_{n}].$$ 
By definition, $A$ is SRMZ-detectable. The other implication 
follows from the reversed argument and Lemma \ref{xxlem1.5}.
\end{proof}

However, there exists a commutative algebra which is 
SRMZ-retractable, but not SMZ-retractable. The following 
example is borrowed from \cite[Example 1]{As}, see also 
\cite[Example 3.3]{LeWZ}.

\begin{example}
\label{xxex1.9}
Let $A=k[x,y]/(x^{2}, y^{2}, xy)$. Then $A$ is SRMZ-retractable. 
Furthermore, it is SMZ-detectable and SRMZ-detectable, 
but neither strongly retractable nor SMZ-retractable.
\end{example}

\section{Morita Cancellation}
\label{xxsec2}

This section concerns Morita cancellative properties.
We also prove some of the results stated in the introduction.
The first result, namely, Theorem \ref{xxthm0.4}, is about 
universally Morita cancellation whose proof is essentially 
adopted from \cite{BZ1}. Let $\GKdim A$ denote the 
Gelfand-Kirillov dimension of an algebra $A$. We refer 
the reader to \cite{KL} and Section \ref{xxsec3} for 
the basic definitions and properties of Gelfand-Kirillov 
dimension. 

\begin{proof}[Proof of Theorem \ref{xxthm0.4}]
Let $B$ be another algebra. Let $R$ be an affine commutative 
domain with an ideal $I\subset R$ such that $R/I=\Bbbk$. Suppose that 
$$\mathcal{E}\colon M(A\otimes_{\Bbbk} R)\longrightarrow M(B\otimes_{\Bbbk} R)$$ 
is an equivalence of abelian categories. By \eqref{E1.2.1}, 
$\mathcal{E}$ induces an isomorphism 
$$\omega \colon Z(A\otimes_{\Bbbk} R)\cong Z(B\otimes_{\Bbbk} R)$$ 
between the centers. Since $Z(A)=\Bbbk$, we obtain that 
$$R=Z(A)\otimes_{\Bbbk}R=Z(A\otimes_{\Bbbk} R)\cong 
Z(B\otimes_{\Bbbk} R)=Z(B)\otimes_{\Bbbk} R.$$ 
As a result, we have that $R\cong Z(B)\otimes_{\Bbbk} R$. 
In particular, $Z(B)$ is a commutative domain. Due to a 
consideration of the GKdimension, we see that $\GKdim Z(B)=0$, 
regarded as a $\Bbbk$-algebra. Thus $Z(B)$ is indeed a field. 
We have that $Z(B)=\Bbbk$ due to the fact that there is an 
ideal $I\subset R$ such that $R/I=\Bbbk$. Consequently, we 
have that $Z(B\otimes_{\Bbbk} R)=R$. As a result, $\omega$ 
is an isomorphism from $R\longrightarrow R$ which implies 
that $R/\omega(I)=\Bbbk$. Note that 
$A\cong (A\otimes_{\Bbbk} R)/I$ is Morita equivalent to 
$(B\otimes_{\Bbbk} R)/(\omega(I))\cong B$ \cite[Lemma 2.1(5)]{LuWZ}. 
Thus, we have proved that $A$ is universally Morita cancellative.
\end{proof}

The following result is a re-statement of Theorem \ref{xxthm0.5}.
It is an analogue of \cite[Theorem 1]{BR} and serves as an 
improvement of \cite[Theorem 4.2]{LeWZ}. Note that our result 
does not require the strongly Hopfian assumption. We should 
mention that \cite[Theorem 1]{BR} deals with the cancellation 
problem in the category of rings; but the idea of its proof 
carries over word in word for $\Bbbk$-algebras.
\begin{theorem}
\label{xxthm2.1}
Let $A$ be an algebra with center $Z$. Suppose either
\begin{enumerate}
\item[(1)]
$Z$ or $Z/N(Z)$ is strongly retractable, or
\item[(2)]
$Z$ or $Z/N(Z)$ is strongly detectable.
\end{enumerate}
Then $Z$ and $A$ are strongly cancellative and strongly 
Morita cancellative.
\end{theorem}

\begin{proof} For the assertions concern $Z$, it suffices to
take $A=Z$. So it is enough to prove the assertions for 
$A$. We only prove that $A$ is strongly Morita 
cancellative. The proof of strongly cancellative property is 
similar, and therefore is omitted.

Under the hypothesis of (1), by Lemma \ref{xxlem1.6}, $A$ is 
SMZ-detectable. If $Z$ is strongly detectable (part of the
hypothesis in (2)), it is clear that $A$ is SMZ-detectable. 
If $Z/N(Z)$ is strongly detectable, by Lemma \ref{xxlem1.5},
$Z$ is strongly detectable. Therefore in all cases, we 
conclude that $A$ is SMZ-detectable. 

Let 
$$\mathcal{E}\colon M(A[s_{1}, \cdots, s_{n}])\longrightarrow 
M(B[t_{1}, \cdots, t_{n}])$$ be an equivalence of abelian 
categories. Then $\mathcal{E}$ induces an algebra isomorphism 
$$\omega \colon Z(A)[s_{1}, \cdots, s_{n}]\longrightarrow 
Z(B)[t_{1}, \cdots, t_{n}].$$ 

Since $A$ is SMZ-detectable, $Z(B)[t_{1}, \cdots, t_{n}]=
Z(B)[f_{1}, \cdots, f_{n}]$ where $f_i=\omega(s_i)$ for 
$i=1,\cdots,n$. Let $I$ be the ideal of $A[s_{1}, \cdots, s_{n}]$ 
generated by $s_{1}, \cdots, s_{n}$. Then $\omega(I)
=B[t_{1}, \cdots, t_{n}]
(f_{1}, \cdots, f_{n})=B[f_{1}, \cdots, f_{n}](f_{1}, \cdots, f_{n})$. 
As a result, we have that $A\cong A[s_{1}, \cdots, s_{n}]/(A[s_{1}, \cdots, s_{n}]I)$ 
which is Morita equivalent to $B[f_{1}, \cdots, f_{n}]/(B[f_{1}, \cdots, f_{n}]
\omega(I))\cong B$. That is, $A$ is Morita equivalent to $B$. 
Therefore, $B$ is strongly Morita cancellative. 
\end{proof}

Next we mention some easy consequences.

\begin{corollary}
\label{xxcor2.2}
Let $A$ be an algebra with a center $Z$. Suppose one of the following holds.
\begin{enumerate}
\item[(1)]
Either $Z$ or $Z/N(Z)$ is an integral domain of transcendence degree one 
over a subfield of $Z$ and is not isomorphic to $\Bbbk^{\prime}[x]$ for 
any field extension $\Bbbk\subseteq \Bbbk^{\prime}\subseteq Z$.
\item[(2)]
$Z$ is an integral domain with nonzero Jacobson radical.
\end{enumerate}
Then $Z$ or $Z/N(Z)$ is strongly retractable. As a consequence, $A$ is 
strongly cancellative and strongly Morita cancellative.
\end{corollary}

\begin{proof} The consequence follows form Theorem \ref{xxthm2.1}. It remains 
to show that $Z$ or $Z/N(Z)$ strongly retractable.

(1) This is \cite[Example 2.2]{LeWZ}.

(2) It follows from \cite[Statement 1.10 on Page 317]{AEH}.
\end{proof}

Now we prove Corollary \ref{xxcor0.6} below.

\begin{proof}[Proof of Corollary \ref{xxcor0.6}]
(1) It follows from the proof of 
\cite[Lemma 2.3]{LeWZ} that $Z$ is strongly retractable. 
The assertion follows from Theorem \ref{xxthm2.1}.

(2) By \cite[Theorem 2]{BR}, a von Neumann regular algebra $Z/N(Z)$ is 
strongly retractable. The assertion follows from Theorem 
\ref{xxthm2.1}. 

(3) By \cite[Theorem 3]{BR}, $Z$ is strongly retractable. 
The assertion follows from Theorem \ref{xxthm2.1}. 
\end{proof}

\begin{remark}
\label{xxrem2.3} Theorem \ref{xxthm2.1} and Corollary \ref{xxcor0.6} have many 
applications. Here is a partial list.
\begin{enumerate}
\item[(1)]
In view of Corollary \ref{xxcor0.6}(2), if $A$ is 
von Neumann regular, then the center $Z$ will also be
von Neumann regular. By Corollary \ref{xxcor0.6}(2),
$A$ is strongly cancellative and strongly Morita cancellative. 
\item[(2)]
If $Z$ is of Krull dimension zero, then $Z/N(Z)$ is
von Neumann regular. By Corollary \ref{xxcor0.6}(2)
again, $A$ is strongly cancellative and strongly Morita 
cancellative. 
\item[(3)]
If $A$ is a finite direct product of simple algebras, then 
$Z$ is a finite product of fields. Thus $Z$ has Krull dimension zero.
By the above comment, $A$ is strongly cancellative and strongly Morita cancellative. 
\item[(4)]
By the proof of \cite[Lemma 2.3]{LeWZ} or \cite[Statement 2.1, p. 320]{AEH},
the Laurent polynomial algebra $\Bbbk[x_{1}^{\pm 1}, \cdots, x_{m}^{\pm}]$ 
is strongly retractable. If $Z$ or $Z/N(Z)$ is isomorphic to the 
Laurent polynomial algebra, then $A$ is strongly cancellative and strongly 
Morita cancellative by Theorem \ref{xxthm2.1}.
\end{enumerate}
\end{remark}

We will explore some skew cancellative properties [Definitions
\ref{xxdef0.7} and \ref{xxdef0.8}] when the algebra $A$
has ``enough'' invertible elements in Section \ref{xxsec4}.

\section{GKdimension and Homological transcendence degree}
\label{xxsec3}

\subsection{GKdimension}
\label{xxsec3.1}
Let $A$ be an algebra over $\Bbbk$. The 
{\it Gelfand-Kirillov dimension} (or {\it GKdimension} for short) of
$A$ is defined to be
\begin{equation}
\notag
\GKdim A:=\sup_{V} \limsup_{n\to\infty} \left(\log_{n} 
(\dim_{\Bbbk} V^n)\right)
\end{equation}
where $V$ runs over all finite dimensional subspaces of $A$.
We refer the reader to \cite{KL} for more details. Next we prove or review
some preliminary results concerning the GKdimension of Ore extensions.

Let $\sigma$ be an automorphism of $A$. Recall that $\sigma$ is called 
{\it locally algebraic} if every finite dimensional subspace 
of $A$ is contained in a $\sigma$-stable finite dimensional subspace 
of $A$. If $A$ is affine, then $\sigma$ is locally algebraic if and only if 
there is a $\sigma$-stable finite dimensional generating subspace.

\begin{lemma}
\label{xxlem3.1} 
Let $A$ be an affine algebra over $\Bbbk$.
\begin{enumerate}
\item[(1)] 
Let $B:=A[t;\sigma,\delta]$
be an Ore extension of $A$. If $\sigma$ is locally algebraic, 
then $\GKdim B=\GKdim A+1$.
\item[(2)]
Let $B$ be an iterated Ore extension $A[t_1;\sigma_1,\delta_1]\cdots
[t_n;\sigma_n, \delta_n]$ where each $\sigma_i$ is 
locally algebraic. Then $\GKdim B=\GKdim A+n$.
\item[(3)]
Let $B$ be an iterated Ore extension $A[t_1;\delta_1]\cdots
[t_n;\delta_n]$. Then $\GKdim B=\GKdim A+n$.
\end{enumerate}
\end{lemma}

\begin{proof} (1) We may assume that $1\in V$.
Let $W$ be any finite dimensional generating subspace of $A$ 
with $1\in W$. Since $V$ generates $A$, $W\subseteq V^n$ for some 
$n$. Without loss of generality, we can assume that $W=V$. 
Since $V$ generates $A$, we have $\delta(V)\subseteq V^m$ 
for some $m$. Now the assertion follows from \cite[Lemma 4.1]{Zh}.

(2) This follows from induction and part (1).

(3) This is a special case of part (2) by setting $\sigma_i$
to be the identity. 
\end{proof}

The reader is referred to \cite[p.74]{KL} for the definition of a filtered algebra.
The following lemma is similar to \cite[Lemma 3.2]{BZ1}.

\begin{lemma}
\label{xxlem3.2}
Let $Y$ be a filtered algebra with an ${\mathbb N}$-filtration 
$\{F_i Y\}_{i\geq 0}$. Assume that the associated graded algebra
$\gr Y$ is an ${\mathbb N}$-graded domain. Suppose $Z$ is a 
subalgebra of $Y$ and let $Z_0=Z\cap F_0 Y$. If $\GKdim Z=\GKdim Z_0<\infty$,
then $Z=Z_0$.
\end{lemma}

\begin{proof} Suppose $Z$ strictly contains the subalgebra $Z_0$. 
There is a natural filtration on $Z$ induced from $Y$ by taking
$F_i Z:=Z\cap F_i Y$. As a result, $\gr Z$ is a subalgebra of $\gr Y$. 
By \cite[Lemma 6.5]{KL},
$$\GKdim Z\geq \GKdim \gr Z\geq \GKdim F_0 Z=\GKdim Z_0=\GKdim Z.$$
Since $\gr Z$ is an ${\mathbb N}$-graded subalgebra of $\gr Y$ that
strictly contains $Z_0=F_0 Z$, there is an element $a\in \gr Z$ of
positive degree. Considering the grading, we see that
$$Z_0+ Z_0 a+ Z_0 a^2 +\cdots$$
is a direct sum contained in $\gr Z$. As a result, we obtain that
$$\GKdim \gr Z\geq \GKdim (\gr Z)_0+1=\GKdim Z_0+1,$$
which yields a contradiction. Therefore $Z=Z_0$.
\end{proof}

The above lemma has an immediate consequence.

\begin{proposition}
\label{xxpro3.3}
Let $Y$ be an iterated Ore extension $A[t_1;\sigma_1,\delta_1]\cdots
[t_n;\sigma_n, \delta_n]$ of a domain $A$. Let $B$ be a subalgebra of $Y$
containing $A$. If $\GKdim A=\GKdim B<\infty$, then $A=B$.
\end{proposition}

\begin{proof} Let $m\leq n$ be the minimal integer
such that $B\subseteq A[t_1;\sigma_1,\delta_1]\cdots
[t_m;\sigma_m, \delta_m]$. It remains to show that
$m=0$. Suppose on the contrary that $m\geq 1$. Let 
$$Y=A[t_1;\sigma_1,\delta_1]\cdots
[t_m;\sigma_m, \delta_m]
\quad {\text{and}}\quad
Y_0=A[t_1;\sigma_1,\delta_1]\cdots
[t_{m-1};\sigma_{m-1}, \delta_{m-1}].$$
Define an ${\mathbb N}$-filtration
on $Y$ by $F_i Y=\sum_{s=0}^{i} Y_0 t_m^{s}$. 
Let $Z=B$ and $Z_0=Z\cap Y_0$. By the 
choice of $m$, we have $Z\neq Z_0$. By the 
hypothesis on the GKdimension, we have
$$\GKdim Z_0\geq \GKdim A=\GKdim B=\GKdim Z\geq \GKdim
Z_0.$$
By Lemma \ref{xxlem3.2}, we have $Z=Z_0$, a contradiction.
Therefore the assertion follows.
\end{proof}

\subsection{Homological transcendence degree}
\label{xxsec3.2}
Another useful invariant is the Homological transcendence degree
introduced in \cite{YZ}. Recall from \cite[Definition 1.1]{YZ}
that the {\it Homological transcendence degree} of a division
algebra $D$ is defined to be
$$\Htr D:=\injdim D\otimes D^{op}$$
where $D^{op}$ is the opposite algebra of $D$. One result of 
\cite[Proposition 1.8]{YZ} is that $\Htr D=n$ if $D$ is a stratiform 
division algebra of stratiform length $n$. We say $A$ is 
stratiform if $A$ is Goldie prime and 
the ring of its fractions, denoted by $Q(A)$, is stratiform.
As an immediate consequence, we have

\begin{lemma}
\label{xxlem3.4}
Let $A$ be a noetherian domain that is stratiform. 
If $B$ is an $n$-step iterated Ore extension of $A$, then 
$\Htr Q(B)=\Htr Q(A)+n$.
\end{lemma}

There is a variety of examples which are stratiform algebras; 
and the following are some typical examples (details are 
omitted).

\begin{example}
\label{xxex3.5} 
The following algebras are stratiform.
\begin{enumerate}
\item[(1)]
Affine commutative domains.
\item[(2)]
PI prime algebras that are finitely
generated over its affine center.
\item[(3)]
Skew polynomial algebras and their localizations 
\cite[Example 1.9(g)]{YZ}. 
\item[(4)]
Quantum Weyl algebras as defined next or their localizations.
Let $q\neq 0,1$ be a scalar in $\Bbbk$.
Let $A_{1}^{q}(\Bbbk)$ denote the first quantum Weyl algebra, which is 
a $\Bbbk$-algebra generated by $x, y$ subject to the relation: 
$xy-qyx=1$. 
\item[(5)]
Prime algebras that can be written as a tensor product of algebras
listed as above.
\end{enumerate}
\end{example}

Here is a version of Proposition \ref{xxpro3.3} with the 
GKdimension replaced by the Homological transcendence degree.

\begin{proposition}
\label{xxpro3.6}
Let $A$ be a noetherian stratiform domain.
Let $Y$ be an iterated Ore extension $A[t_1;\sigma_1,\delta_1]\cdots
[t_n;\sigma_n, \delta_n]$. Let $B$ be a subalgebra of $Y$
containing $A$ such that it is stratiform. 
If $\Htr Q(A)=\Htr Q(B)$, then $A=B$.
\end{proposition}

\begin{proof}
Suppose on the contrary that $A\neq B$. 
By the proof of Lemma \ref{xxlem3.2},
there is an element $a\in B$ such that $Q(A)+Q(A)a+Q(A)a^2+
\cdots$ is a direct sum. 
This implies that $\dim _{Q(A)}Q(B)$
is infinite. Note that $Q(B)$ is a $(Q(A),Q(B))$-bimodule
that is finitely generated as a right $Q(B)$-module. 
Since $Q(A)$ and $Q(B)$ have the same stratiform length
by \cite[Proposition 1.8]{YZ}, $Q(B)$ is finitely
generated as left $Q(A)$-module by \cite[Theorem 24]{Sc}.
This contradicts the fact that $\dim _{Q(A)}Q(B)$
is infinite. Therefore, we have that $A=B$. 
\end{proof}

\section{Divisor Subalgebras and Skew cancellations}
\label{xxsec4}

Recall from \cite{BZ1, LeWZ, LuWZ} that the (strong) retractability implies
the (strong) cancellation. It is clear that the (strong) retractability 
implies the (strong) $Z$-retractability, which in turn implies (strong) 
Morita cancellation, also see Section \ref{xxsec2} and Theorem \ref{xxthm0.5}. 
In this section we continue our investigation on the (strong) 
retractability with a twist. 

As asked in \cite{LeWZ}, one would like to know how localizations affect 
cancellative properties. Indeed, cancellative and Morita cancellative 
properties are preserved under localizations for many families of algebras.
We add some results along this line; and we will address the problem 
in a forthcoming paper later on. 

Note that even if the discriminant $d$ of an algebra $A$ is dominating or effective, 
the discriminant of $A_{d}$ becomes invertible, where $A_{d}$ is the localization of $A$ 
with respect to the Ore set $\{d^{i}\mid i\geq 0\}$. As a result, the discriminant of $A_{d}$ 
is neither dominating nor effective. However, since the element $d$ is invertible in $A_{d}$, 
it will be sent to an invertible element by any 
$\Bbbk$-algebra isomorphism $\phi \colon A_{d}[s_{1}, \cdots, s_{n}]
\longrightarrow B[t_{1}, \cdots, t_{n}]$. So we may still be able to 
prove that the algebra $A_{d}$ is strongly retractable in many situations. 

The point of this section is that we can do more. Namely, we can prove
a version of the strong retractability even in the Ore extension setting.
The main idea is to utilize the notion of divisor subalgebras introduced in \cite{CYZ1}.

We first recall the definition of a divisor subalgebra. Let $A$ be a 
domain. Let $F$ be a subset of $A$. Let $Sw(F)$ 
denote the set of $g\in A$ such that $f=agb$ for some $a, b\in A$ and 
$0\neq f\in F$. That is, $Sw(F)$ consists of all the subwords of the 
elements in $F$. The following definition is quoted from \cite{CYZ1}.

\begin{definition}
\label{xxdef4.1}
Let $F$ be a set of elements in a domain $A$.
\begin{enumerate}
\item[(1)] 
We set $D_{0}(F)=F$ and inductively define $D_{n}(F)$ for $n\geq 1$ as the 
$\Bbbk$-subalgebra of $A$ generated by $Sw(D_{n-1}(F))$. The 
subalgebra ${\mathbb D}(F)=\cup_{n\geq 0} D_{n}(F)$ is called the {\it divisor 
subalgebra} of $A$ generated by $F$. If $F$ is the singleton $\{f\}$, 
we simply write ${\mathbb D}(\{f\})$ as ${\mathbb D}(f)$.
If we need to indicate the ambient algebra $A$, we write ${\mathbb D}(F)$
as ${\mathbb D}_A(F)$.
\item[(2)] 
If $f=d(A/Z)$ (if the discriminant $d(A/Z)$ indeed exists), we call 
${\mathbb D}(f)$ the {\it discriminant-divisor subalgebra} of $A$ or $DDS$ of 
$A$, and write it as $\mathbb{D}(A)$.
\end{enumerate} 
\end{definition}

We now define some elements which play the same role as the 
dominating or effective elements in the study of cancellative 
properties. 

\begin{definition}
\label{xxdef4.2}
Let $F$ be a set of elements in an algebra $A$ which is a domain. 
\begin{enumerate}
\item[(1)]
We say $F$ is a {\it controlling set} if ${\mathbb D}(F)=A$. 
\item[(2)]
An element $0\neq f\in A$ is called {\it controlling} if ${\mathbb D}(f)=A$.
\end{enumerate} 
\end{definition}

Next we give some examples of controlling elements. For an algebra $A$, let
$A^{\times}$ denote the set of invertible elements in $A$. 

\begin{example}
\label{xxex4.3}
Let $q\neq 0,1$ be a scalar in $\Bbbk$.
\begin{enumerate}
\item[(1)]
Let $A_{1}^{q}(\Bbbk)$ be the first quantum Weyl algebra defined as in
Example \ref{xxex3.5}(4). Set $z=xy-yx$, then $xy=\frac{qz-1}{q-1}$ and 
$yx=\frac{z-1}{q-1}$. It is obvious that $z$ is controlling, 
dominating, and effective in $A_{1}^{q}(\Bbbk)$.
\item[(2)] 
We can localize $A_{1}^{q}(\Bbbk)$ with respect to the Ore set 
generated by $z$ and denote the localization by $B_{1}^{q}(\Bbbk)$. 
Since $z$ is a controlling element in $A_{1}^{q}(\Bbbk)$, we have 
${\mathbb D}_{B_{1}^{q}(\Bbbk)}(1)=B_{1}^{q}(\Bbbk)$. Note that the 
center of $B_{1}^{q}(\Bbbk)$ is $\Bbbk$ if $q$ is not a root of 
unity. If $q$ is a root of unity of order $l$, then the center 
of $B_{1}^{q}(\Bbbk)$ is isomorphic to $Z:=\Bbbk[x^{l}, y^{l}, z^{l}]/I$ 
where $I=(z^{l}[1-(1-q)^{l}x^{l}y^{l}]-1)$ by \cite[Proposition 3.2]{LY}.
It is also clear that ${\mathbb D}_Z(1)=Z$.
\item[(3)]
The above example can be extended to higher ranks with 
multiparameters in both root of unity and non-root of unity cases.
\item[(4)]
Let $D=\Bbbk [h^{\pm 1}]$ with a $\Bbbk$-algebra automorphism $\phi$ 
defined by $\phi(h)=qh$ for some $q\in \Bbbk^{\times}$. Let $0\neq a\in D$ 
and denote the generalized Weyl algebra by $A=D(a, \sigma)$, which is the 
$\Bbbk$-algebra generated by $x, y, h^{\pm 1}$ subject to the relations:
\[
xy=a(qh), \quad yx=a(h), xh=qha, yh=q^{-1}hy.
\]
Then ${\mathbb D}_{A}(1)=A$.
\item[(5)]
Fix a positive integer $n\geq 2$, let $q$ be a set of nonzero 
scalars $\{q_{ij}\mid 1\leq i< j\leq n\}$.
A quantum torus (or quantum Laurent polynomial algebra) $T^q_{n}(\Bbbk)$ is 
generated by generators
$\{x_1^{\pm 1},\cdots, x_n^{\pm 1}\}$ and subject to the relations
$$x_j x_i= q_{ij} x_i x_j$$
for all $i<j$. Note that $T^q_n(\Bbbk)$ is PI if and only if all $q_{ij}$ are roots 
of unity. It is clear that ${\mathbb D}_{T^q_n(\Bbbk)}(1)=T^q_n(\Bbbk)$.
This well-known that the center of $T^q_n(\Bbbk)$ is isomorphic to a 
commutative Laurent polynomial algebra of possibly lower rank. Note that
the quanutm torus is a localization of a skew polynomial ring given in 
Example \ref{xxex3.5}(3). 
\item[(6)]
The tensor products or twisted tensor products $A$ of these examples also 
satisfy ${\mathbb D}_{A}(1)=A$.
\end{enumerate}
\end{example}

The proof of the following lemma is easy and omitted.

\begin{lemma}
\label{xxlem4.4}
Let $A$ be a domain and $F$ be a subset of $A$. We have the following
\begin{enumerate}
\item[(1)]
${\mathbb D}({\mathbb D}(F))={\mathbb D}(F)$.
\item[(2)]
${\mathbb D}_{{\mathbb D}(F)}(F)={\mathbb D}_A(F)$.
\item[(3)]
Suppose $B$ is a subalgebra of $A$ containing $F$. Then 
${\mathbb D}_{B}(F) \subseteq {\mathbb D}_A(F)$.
\item[(4)]
Let $C$ be the Ore extension $A[t;\sigma,\delta]$.
Then ${\mathbb D}_C(F)\subseteq A$.
As a consequence, ${\mathbb D}_C(F)={\mathbb D}_A(F)$.
\item[(5)]
Let $C$ be an iterated Ore extension of $A$. Then ${\mathbb D}_C(F)={\mathbb D}_A(F)$.
\item[(6)]
Let $\phi: A\to B$ be an injective algebra homomorphism.
Then ${\mathbb D}_A(F)\subseteq {\mathbb D}_B(\phi(F))$.
If $\phi$ is an isomorphism, then ${\mathbb D}_A(F)= 
{\mathbb D}_B(\phi(F))$.
\item[(7)]
Let $\overline{A}$ be an iterated Ore extension 
$A[t_1,\sigma_1,\delta_1][t_2,\sigma_2, \delta_2]\cdots 
[t_n,\sigma_n,\delta_n]$ and $\overline{B}$ be an iterated 
Ore extension 
$B[t'_1,\sigma'_1,\delta'_1][t'_2,\sigma'_2, \delta'_2]\cdots 
[t'_n,\sigma'_n,\delta'_n]$.
Suppose 
$$\phi: \overline{A}
\to \overline{B}$$
is an isomorphism. Then $\phi({\mathbb D}_{\overline{A}}(1))
={\mathbb D}_{\overline{B}}(1)\subseteq B$.
\item[(8)]
Suppose $A$ and $B$ are algebras such that $A\otimes_{\Bbbk} B$ is a
domain. If ${\mathbb D}_{A}(1)=A$ and ${\mathbb D}_{B}(1)=B$,
then ${\mathbb D}_{A\otimes_{\Bbbk} B}(1)=A\otimes_{\Bbbk} B$.
\item[(9)]
If $A$ is a finitely generated left (or right) module over 
${\mathbb D}_{A}(1)$, then ${\mathbb D}_{A}(1)=A$.
\end{enumerate}
\end{lemma}

Now we are ready to prove Theorems \ref{xxthm0.9}
and \ref{xxthm0.10}.

\begin{proof}[Proof of Theorem \ref{xxthm0.9}]
Let $A$ be an affine domain of finite GKdimension.
Let $\overline{A}$ be an iterated Ore extension 
$A[t_1;\sigma_1,\delta_1][t_2;\sigma_2, \delta_2]\cdots 
[t_n;\sigma_n,\delta_n]$ where each $\sigma_i$
is locally algebraic. By Lemma \ref{xxlem3.1}(2),
$\GKdim \overline{A}=\GKdim A+n$.
Now let $B$ be an algebra and $\overline{B}$ be
an iterated Ore extension 
$B[t'_1;\sigma'_1,\delta'_1][t'_2;\sigma'_2, \delta'_2]\cdots 
[t'_n;\sigma'_n,\delta'_n]$.
Suppose $\phi: \overline{A}\to \overline{B}$
is an algebra isomorphism. It remains to show that
$A\cong B$. By the hypothesis and Lemma 4.4(7), 
$\phi(A)=\phi({\mathbb D}_{\overline{A}}(1))=
{\mathbb D}_{\overline{B}}(1)\subseteq B$. 
Let $B'$ denote $\phi^{-1}(B)$. Then $A\subseteq B'$
and $\GKdim B'=\GKdim B$. By the definition of 
$\overline{B}$, we have 
$$\GKdim B'=\GKdim B\leq \GKdim \overline{B}-
n=\GKdim \overline{A}-n=\GKdim A\leq \GKdim B.$$
Therefore, $\GKdim A=\GKdim B'<\infty$. By 
Proposition \ref{xxpro3.3}, we have $A=B'$. 
This implies that $\phi: A\to B$ is an isomorphism.
\end{proof}

\begin{proof}[Proof of Theorem \ref{xxthm0.10}]
Let $A$ be a noetherian domain that is 
stratiform. Let $\overline{A}$ be an iterated Ore extension 
$A[t_1;\sigma_1,\delta_1][t_2;\sigma_2, \delta_2]\cdots 
[t_n;\sigma_n,\delta_n]$. By Lemma \ref{xxlem3.4},
$\Htr \overline{A}=\Htr A+n$.
Now let $B$ be another noetherian domain that is 
stratiform and $\overline{B}$ be
an iterated Ore extension 
$B[t'_1;\sigma'_1,\delta'_1][t'_2;\sigma'_2, \delta'_2]\cdots 
[t'_n;\sigma'_n,\delta'_n]$.
Suppose $\phi: \overline{A}\to \overline{B}$
is an algebra isomorphism. We need to show that
$A\cong B$. By the hypothesis and Lemma 4.4(7), 
$\phi(A)=\phi({\mathbb D}_{\overline{A}}(1))=
{\mathbb D}_{\overline{B}}(1)\subseteq B$. 
Let $B'$ denote $\phi^{-1}(B)$. Then $A\subseteq B'$
and $\Htr Q(B')=\Htr Q(B)$. By the definition of 
$\overline{B}$, we have 
$$\begin{aligned}
\Htr Q(B')&=\Htr Q(B)=\Htr Q(\overline{B})-n \qquad {\text{by Lemma \ref{xxlem3.4}}}\\
&=\Htr Q(\overline{A})-n=\Htr Q(A)\qquad {\text{by Lemma \ref{xxlem3.4}}}\\
&\leq \Htr Q(B).
\end{aligned}
$$
Therefore, $\Htr Q(A)=\Htr Q(B')<\infty$. According to Proposition \ref{xxpro3.6}, 
we have that $A=B'$. That is, $\phi: A\to B$ is indeed an isomorphism.
\end{proof}

In the rest of this section, we will prove that a class of 
simple algebras are strongly $\sigma$-invariant, but 
might not be $\delta$-cancellative, see [Example \ref{xxex5.5}(1)]. 
First of all, we need a lemma.

\begin{lemma}
\label{xxlem4.5}
Let $\overline{B}$ be an iterated Ore extension 
$B[t_1;\sigma_1]\cdots[t_n;\sigma_n]$ of an algebra $B$. 
If $A$ is a simple factor ring of $\overline{B}$ such 
that $A^{\times}=\Bbbk^{\times}$ with the quotient map 
denoted by $\pi \colon \overline{B}\longrightarrow A$,
then the image $\pi(t_{i})$ of each $t_i$ in $A$ is a 
scalar and $\pi(B)=A$. Furthermore, if $B$ is a simple 
algebra, then $A\cong B$.
\end{lemma}

\begin{proof} We use an induction argument. If $n=0$, it is 
trivial. Now we assume that $n>0$. Let $\pi$ denote the 
quotient map from $\overline{B}$ onto $A$. Since $t_n$ is 
normal in $\overline{B}$,
so is $\pi(t_n)$ in $A$. Since $A$ is simple, $\pi(t_n)$
is either zero or invertible in $A$. Since $A^{\times}=\Bbbk^{\times}$,
$\pi(t_n)$ is a scalar in $\Bbbk$. As a result, we have that 
\[
\pi(B[t_1;\sigma_1]\cdots[t_{n-1};\sigma_{n-1}])=A.
\] 
By induction, we have that $\pi(t_i)$ is a scalar in $A$ 
for $1\leq i\leq n$ and $\pi(B)=A$. If $B$ is simple, then $B\cong A$.
\end{proof}

The next result establishes the strongly $\sigma$-cancellative 
property for many simple algebras such as the first Weyl 
algebra, which is the $\Bbbk$-algebra generated by $x,y$ 
subject to the relation $xy-yx=1$. 

\begin{theorem}
\label{xxthm4.6}
Let $A$ be a right {\rm{(}}resp. left{\rm{)}} noetherian simple 
domain such that $A^{\times}=\Bbbk^{\times}$. Then $A$ is strongly 
$\sigma$-cancellative.
\end{theorem}

\begin{proof}
Suppose
$$\phi: \overline{A}:=A[s_1;\sigma_1]\cdots[s_n;\sigma_n]\to
B[t_1;\tau_1]\cdots[t_n;\tau_n]=:\overline{B}$$
is an isomorphism for another algebra $B$. Here $\sigma_i$
and $\tau_i$ are automorphisms of appropriate algebras.
Then $\overline{B}$ is a noetherian domain, and consequently $B$ 
is a noetherian domain. By \cite[Theorem 15.19]{GoW}, $A$ and $B$ have 
the same Krull dimension. Let $I$ be the ideal of
$A$ generated by $\{s_i\}_{i=1}^n$. Then $A\cong \overline{A}/I$.
Let $J=\phi(I)$. Then $\overline{B}/J (\cong A)$ is simple and every 
invertible element in $\overline{B}/J$ is a scalar by 
hypothesis (1). Let $\phi$ also denote the induced isomorphism
$$A\to \overline{B}/J.$$
Let $\pi$ be the map from $\overline{B}$ onto $\overline{B}/J$.
By Lemma \ref{xxlem4.5},
$\pi(t_i)$ is a scalar in $\Bbbk$ for each $i$. Then 
$$\pi(B)=\pi(\overline{B})=\overline{B}/J.$$
Therefore, $\phi^{-1}\circ \pi|_B: B\longrightarrow \overline{B}/J
\xrightarrow{\phi^{-1}}A$ is a surjective algebra homomorphism.
Since $B$ is a domain with $\Kdim B=\Kdim A$, we obtain 
that $\phi^{-1}\circ \pi|_B$ is an isomorphism and that $B\cong \overline{B}/J 
\cong A$ as desired.
\end{proof}

\section{Comments, examples, remarks and questions}
\label{xxsec5}
In this section we give some isolated results, comments, examples, 
remarks and open questions. 

\subsection{Cancellative property of some infinitely generated algebras}
\label{xxsec5.1}
In most of the results proved in \cite{LeWZ, LuWZ}, we have assumed the algebras are 
either affine or noetherian or having finite GKdimension. In this 
subsection we make some comments on the cancellation property for some 
algebras of infinite GKdimension. The following lemma generalizes 
\cite[Lemma 2]{BR} and \cite[Corollary 1]{CE}. 

\begin{lemma}
\label{xxlem5.1}
Let $A$ and $B$ be $\Bbbk$-algebras and 
$$\phi \colon A[s_{1}, \cdots, s_{n}]\longrightarrow 
B[t_{1}, \cdots, t_{n}]$$ 
be an algebra isomorphism.
\begin{enumerate}
\item[(1)]
If $\phi(A)\subseteq B$, then $\phi(A)=B$.
\item[(2)]
If $\phi(Z(A))\subseteq B$, then $A\cong B$.
\end{enumerate}
\end{lemma}

\begin{proof}
(1) We have the following restriction of the isomorphism $\phi$ to 
the respective 
centers of $A[s_{1}, \cdots, s_{n}]$ and $B[t_{1}, \cdots, t_{n}]$:
\[
\phi\colon Z(A)[s_{1}, \cdots, s_{n}]\xrightarrow{\cong}
Z(B)[t_{1}, \cdots, t_{n}].
\]
Note that $Z(B)=B\cap Z(B)[t_1,\cdots,t_n]$.
Since $\phi$ is an isomorphism and, by the hypothesis $\phi(A)\subseteq B$, 
we obtain that $\phi(Z(A))\subseteq Z(B)$. By 
\cite[Lemma 2]{BR}, we have that $\phi(Z(A))=Z(B)$. Let 
$f_{i}:=\phi(s_{i})$ for $1\leq i\leq n$. We have that 
$$Z(B)[t_{1}, \cdots, t_{n}]=Z(B)\{f_{1}, \cdots, f_{n}\}.$$ 
According to Lemma \ref{xxlem1.2}(2), we have $Z(B)[t_{1}, \cdots, t_{n}]
=Z(B)[f_{1}, \cdots, f_{n}]$. Using Lemma \ref{xxlem1.2}(1), we 
can further conclude that $B[t_{1}, \cdots, t_{n}]=B[f_{1}, \cdots, f_{n}]$, 
where $f_{1}, \cdots, f_{n}$ are considered as central 
indeterminates. Denote by $\tau$ the $B$-automorphism of 
$B[t_{1}, \cdots, t_{n}]$, which is defined by setting $\tau(f_{i})=t_{i}$. 
Then we have an isomorphism
$$\tau\circ \phi \colon A[s_{1}, \cdots, s_{n}]
\longrightarrow B[t_{1}, \cdots, t_{n}]$$ 
with $\tau\circ \phi(s_i)=t_i$ for all $i$ 
and $(\tau\circ \phi)(A)\subseteq B$. As a result, we have that $\phi(A)=B$.

(2) Similar to the proof of part (1), we have the following induced isomorphism
\[
\phi\colon Z(A)[s_{1}, \cdots, s_{n}]\longrightarrow 
Z(B)[t_{1}, \cdots, t_{n}].
\]
Since $Z(B)=B\cap Z(B)[t_1,\cdots,t_n]$, the hypothesis $\phi(Z(A))\subseteq B$ implies that 
$\phi(Z(A))\subseteq Z(B)$. By \cite[Lemma 2]{BR}, we have that $\phi(Z(A))=Z(B)$. Let 
$f_{i}:=\phi(s_{i})$ for $1\leq i\leq n$. Similar to the proof of 
part (1), we have that 
\[
B[t_{1}, \cdots, t_{n}]=B[f_{1}, \cdots, f_{n}]
\] 
where $f_{1}, \cdots, f_{n}$ are considered as central indeterminates. Going back to the isomorphism
$$\phi: \overline{A}:=A[s_{1}, \cdots, s_{n}]\longrightarrow 
B[t_{1}, \cdots, t_{n}]=
B[f_1,\cdots,f_n]=:\overline{B}, $$
one sees that $\phi$ maps the ideal of $\overline{A}$
generated by $\{s_i\}_{i=1}^w$ to the ideal of 
$\overline{B}$ generated by $\{f_i\}_{i=1}^n$. Therefore, we have that
$$A\cong \overline{A}/(s_i:i=1,\cdots,n)
\cong \overline{B}/(f_i: i=1,\cdots,n)\cong B.$$
\end{proof}

Combining some ideas in the previous section, we can further have the following result.

\begin{proposition}
\label{xxpro5.2}
Let $A$ be an algebra such that ${\mathbb D}(1)\supseteq Z(A)$.
\begin{enumerate}
\item[(1)]
Then $A$ is strongly cancellative.
\item[(2)]
If $A$ is commutative, then $A$ is strongly retractable.
\end{enumerate}
\end{proposition}

\begin{proof} (1) Let $\phi: \overline{A}=A[s_1,\cdots,s_n]\to
\overline{B}=B[t_1,\cdots,t_n]$ be an isomorphism. By 
Lemma \ref{xxlem4.4}(7), we have that
$$\phi({\mathbb D}(1))=\phi({\mathbb D}_{\overline{A}}(1))
\subseteq {\mathbb D}_{\overline{B}}(1) \subseteq B.$$
Since $Z(A)\subseteq {\mathbb D}(1)$, we have 
$\phi(Z(A))\subseteq B$. Now the assertion follows 
from Lemma \ref{xxlem5.1}(2). 

(2) To prove the second assertion, we repeat the 
above proof and apply Lemma \ref{xxlem5.1}(1).
\end{proof}

By Proposition \ref{xxpro5.2}(2), any commutative algebra $A$ with 
${\mathbb D}(1)=A$ is strongly retractable. For example,
the Laurent polynomial algebra with infinitely many variables
$\Bbbk[x_{i}^{\pm 1}:i=1,2,3, \cdots]$ is strongly retractable. 
As a consequence of Theorem \ref{xxthm0.5}, any algebra with a 
center equal to a finite direct sum of (infinite) Laurent polynomial 
algebras is strongly cancellative and Morita cancellative.
It is obvious that the infinite quantum Laurent polynomial 
algebra $T^q_{\infty}(\Bbbk)$ has this kind of property. Below 
is a similar example.

\begin{example}
\label{xxex5.3}
Let $q$ be a sequence of scalars $\{q_i\}_{i\geq 1}$
and let $A^{q}_{\infty}(\Bbbk)$ be the $\Bbbk$-algebra generated 
by an infinite set $\{x_{i}, y_{i}\}_{i\geq 1}$ subject 
to the relations: 
\begin{eqnarray*}
x_{i}x_{j}=x_{j}x_{i}, \quad y_{i}y_{j}=y_{j}y_{i},\\ 
x_{i}y_{j}=y_{j}x_{i} \, (i\neq j),\quad x_{i}y_{i}-q_{i}y_{i}x_{i}=1
\end{eqnarray*}
for $i,j\in \{1, 2, 3, \cdots\}$. It is clear that $A^q_{\infty}(\Bbbk)$ 
is an infinite tensor product of the quantum Weyl algebras defined 
in Example \ref{xxex4.3}(1). We call this algebra {\it infinite 
quantum Weyl algebra}. It is obvious that $\GKdim A^{q}_{\infty}(\Bbbk)
=\infty$. Set $z_{i}=x_{i}y_{i}-y_{i}x_{i}$ and 
let $S$ be the Ore set generated by the products of $z_{i}$. 
Let $B$ be the localization of $A^{q}_{\infty}(\Bbbk)$ with respect to $S$. 
It is easy to see that ${\mathbb D}_B(1)=B$ and ${\mathbb D}_{Z(B)}(1)=Z(B)$ 
(some details are omitted).
As a consequence of Proposition \ref{xxpro5.2} and Theorem \ref{xxthm0.5},
$B$ is strongly cancellative and strongly Morita cancellative. 
However, it is not known to us whether or not it is (strongly) skew 
cancellative.
\end{example}

\subsection{$\delta$-cancellative property of LND-rigid algebras}
\label{xxsec5.2}
In this subsection we use some ideas in \cite{BHHV} to 
show that every LND-rigid algebra is $\delta$-cancellative.
To save some space, we refer the reader to \cite[Definition 2.1]{BHHV} 
for the definitions concerning LND-rigidity and Makar-Limanov invariants
$ML(A)$.

\begin{theorem}
\label{xxthm5.4}
Suppose that $\Bbbk$ is a field of characteristic zero. 
Let $A$ be an affine $\Bbbk$-domain of finite GKdimension. Suppose
that $ML(A)=A$. Then $A$ is $\delta$-cancellative.
\end{theorem}

\begin{proof} We follow the proof of \cite[Proposition 5.6]{BHHV}. 

Suppose that $\phi: A[s;\delta]\longrightarrow B[t;\delta']$ is 
an isomorphism for another algebra $B$. Here $\delta$ is a 
derivation of $A$ and $\delta'$ is a derivation of $B$. By 
Lemma \ref{xxlem3.1}(1), we have that
$$\GKdim A[s;\delta]=\GKdim A+1<\infty.$$
A similar statement holds true for $B$. Thus we have that $\GKdim B=
\GKdim A$. 

Since $A$ is an affine domain of finite GKdimension, it is an Ore domain. Also note 
that $ML(A)=A$. By \cite[Lemma 5.3]{BHHV}, we have that $ML(A[s;\delta])=ML(A)$, 
which is further equal to $A$. As a result, we have the following 
$$A=ML(A)=ML(A[s;\delta])\stackrel{\phi}\iso ML(B[t;\delta'])\subseteq B.$$
Equivalently, we have that $\phi(A)\subseteq B$. Set $B'=\phi^{-1}(B)$. By 
Proposition \ref{xxpro3.3}, we have that $A=B'$ or equivalently,
$A\cong B$. Therefore, $A$ is $\delta$-cancellative.
\end{proof}

However, not every algebra satisfying the hypotheses in Theorem \ref{xxthm5.4}
is $\sigma$-cancellative. The following is an example along this line, 
see Example \ref{xxex5.5}(2).

\begin{example}
\label{xxex5.5}
Here we consider two affine domains of GKdimension two.
\begin{enumerate}
\item[(1)]
Let $A$ be the first Weyl algebra over a field $\Bbbk$ of characteristic zero. 
Then $A$ is simple with a trivial center $\Bbbk$. By \cite[Proposition 1.3]{BZ1} 
and Theorem \ref{xxthm0.4}, $A$ is universally cancellative and universally 
Morita cancellative. By Theorem \ref{xxthm4.6}, $A$ is strongly $\sigma$-cancellative. 
Now we claim that $A$ is not $\delta$-cancellative. Note that the first Weyl 
algebra $A$ can be written as $\Bbbk\langle x,y\rangle/(xy-yx-1)$.
Let $B=\Bbbk[y,z]$ and let $\delta'$ be the derivation of $B$ defined 
by $\delta'(y)=1$ and $\delta'(z)=0$. Then $A[z;\delta=0]\cong B[x;\delta']$. 
It is clear that $A\not\cong B$. Therefore, $A$ is not $\delta$-cancellative.

\item[(2)]
Let $A$ be a different algebra 
$\Bbbk_{-1}[x,y]=\Bbbk\langle x,y\rangle/(xy+yx)$ which is an
affine noetherian PI domain of GKdimension two with center
$Z=\Bbbk[x^2,y^2]$.
By \cite[Theorem 4.7 and Example 4.8]{BZ1}, $A$ is strongly 
LND-rigid and strongly cancellative. By Theorem 
\ref{xxthm5.4}, $A$ is $\delta$-cancellative. But
$A$ is not $\sigma$ cancellative as $A[z;Id_A]\cong 
\Bbbk[y,z][x;\sigma]$ for some automorphism 
$\sigma$ of the commutative polynomial algebra $\Bbbk[y,z]$.
\end{enumerate}
\end{example}

Figure \ref{E5.5.1} below summarizes the implication relations 
among several types of skew cancellations, where an arrow (resp. dotted arrow) 
means ``implies'' (resp. ``does not implies''). All the implications follow
directly from the definitions.
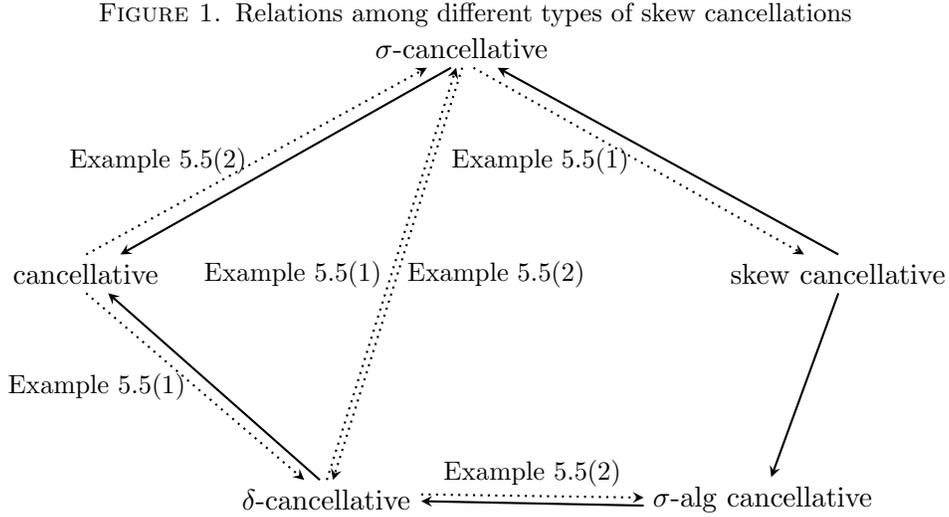
\begin{figure}[ht]
\caption{Relations among different types of skew cancellations}
\label{E5.5.1}
\centering
\tikzstyle{arrow} = [thick, ->, >= stealth]
\begin{tikzpicture}
\node(a){\large$\sigma$-cancellative};
\node[below of=a,node distance=30mm](b){};
\node[left of=b,node distance=50mm](c){\large cancellative};
\node[right of=b,node distance=50mm](d){\large skew cancellative};
\node[below of=b,node distance=30mm](e){};
\node[left of=e,node distance=18mm](f){\large$\delta$-cancellative};
\node[right of=e,node distance=40mm](g){\large$\sigma$-alg cancellative};
\draw[arrow](a.240)--node[anchor=west] {}(c);
\draw[arrow,dotted](c.north)--node[anchor=east] {Example \ref{xxex5.5}(2)}(a);
\draw[arrow,dotted](a.300)--node[anchor=east] {Example \ref{xxex5.5}(1)}(d);
\draw[arrow](d.north)--node[anchor=west] {}(a);
\draw[arrow,dotted](c.south)--node[anchor=east] {Example \ref{xxex5.5}(1)}(f);
\draw[arrow](f.110)--node[anchor=west] {}(c);
\draw[arrow](d.south)--node[anchor=west] {}(g);
\draw[arrow,dotted](f.3)--node[anchor=south] {Example \ref{xxex5.5}(2)}(g);
\draw[arrow](g.183)--node[anchor=north] {}(f);
\draw[arrow,dotted](a.south)--node[anchor=west] {Example \ref{xxex5.5}(2)}(f);
\draw[arrow,dotted](f.north)--node[anchor=east] {Example \ref{xxex5.5}(1)}(a);
\end{tikzpicture}
\end{figure}

Now a natural question to consider is
\begin{question}
\label{xxque5.6}
Let $A$ be an algebra as in Theorem \ref{xxthm5.4},
or specifically, the algebra in Example \ref{xxex5.5}(2).
Or suppose that $A$ is strongly LND-rigid in the sense 
of \cite[Definition 2.3]{BZ1}. 
Is then $A$ strongly $\delta$-cancellative?
\end{question}

\subsection{Derived cancellative property}
\label{xxsec5.3}
First we recall the definition of derived cancellation.
Let $M(A)$ denote the category of all right $A$-modules 
and $D(A)$ be the corresponding derived category of $M(A)$.

\begin{definition}
\label{xxdef5.7}
Let $A$ be an algebra.
\begin{enumerate}
\item[(1)]
The {\it derived cancellative} property of $A$ is defined in the
same way as in Definition \ref{xxdef0.2}(1) by replacing the 
{\sf abelian} categories $M(-)$ with the {\sf triangulated} 
categories $D(-)$. 
\item[(2)]
The {\it strongly derived cancellative} property of $A$ is defined 
in the same way as in Definition \ref{xxdef0.2}(2) by replacing the 
{\sf abelian} categories $M(-)$ with the {\sf triangulated} 
categories $D(-)$. 
\end{enumerate}
\end{definition}

The following is a version of \cite[Corollary 7.3]{LuWZ}
without the strongly Hopfian condition. Its proof is omitted
(see the proof of \cite[Corollary 7.3]{LuWZ}).

\begin{proposition}
\label{xxpro5.8}
Let $A$ be an Azumaya algebra over its center $Z$ 
which has a connected spectrum. Suppose that $Z$ is 
either {\rm{(}}strongly{\rm{)}} detectable or {\rm{(}}strongly{\rm{)}}
retractible, then $A$ is {\rm{(}}strongly{\rm{)}} 
cancellative, {\rm{(}}strongly{\rm{)}} Morita cancellative and 
{\rm{(}}strongly{\rm{)}} derived cancellative.
\end{proposition}

\begin{corollary}
\label{xxcor5.9}
Let $A$ be a domain. If $A$ is Azumaya and 
${\mathbb D}_{Z(A)}(1)=Z(A)$, then $A$ is strongly cancellative, 
strongly Morita cancellative and 
strongly derived cancellative.
\end{corollary}

\begin{proof} By Proposition \ref{xxpro5.2}, $Z$ is 
strongly retractable. The assertion follows from 
Proposition \ref{xxpro5.8}.
\end{proof}

The following is also known due to \cite[Corollary 7.3]{LuWZ}. 
Note that all the algebras involved in the example below are 
Azumaya algebras.

\begin{example}
\label{xxex5.10}
As an immediate consequence of Corollary \ref{xxcor5.9}, 
the following algebras are strongly derived cancellative.  
\begin{enumerate}
\item[(1)]
Localized quantum Weyl algebra $B^q_1(\Bbbk)$ as in 
Example \ref{xxex4.3}(2) where $q$ is a root of unity.
\item[(2)]
Quantum Laurent polynomial algebras as in Example \ref{xxex4.3}(5)
where all $q_{ij}$ are roots of unity.
\item[(3)]
Any finite tensor product of algebras in parts (1) and (2).
\end{enumerate}
\end{example}

\subsection*{Acknowledgments}
X. Tang and X.-G. Zhao would like to thank J.J. Zhang and the department of 
Mathematics at University of Washington for the hospitality during 
their visits. J.J. Zhang was partially supported by the US National Science 
Foundation (No. DMS-1700825). X.-G. Zhao was partially supported by the Characteristic 
Innovation Project of Guangdong Provincial Department of 
Education (2017KTSCX173), the Science and Technology Program of Huizhou City (2017C0404020), 
and the National Science Foundation of Huizhou University (hzu201704, hzu201804).

\providecommand{\bysame}{\leavevmode\hbox to3em{\hrulefill}\thinspace}
\providecommand{\MR}{\relax\ifhmode\unskip\space\fi MR }
\providecommand{\MRhref}[2]{%

\href{http://www.ams.org/mathscinet-getitem?mr=#1}{#2} }
\providecommand{\href}[2]{#2}

\end{document}